\documentclass[12pt]{amsart}
\usepackage{amsfonts,amsthm, amsmath}

\theoremstyle{plain}
\newtheorem*{acknowledgment}{Acknowledgment}

\newtheorem{algorithm}{Algorithm}[section]
\newtheorem{thm}{Thm}

\newtheorem{example}[algorithm]{Example}

\newtheorem{theorem} [algorithm] {Theorem}
\newtheorem{theoremlet}[thm]{Theorem}

\newtheorem{definitionlet}[thm]{Definition}
\newtheorem*{remarknonum}{Remark}

\newtheorem{proposition}[algorithm]{Proposition}
\newtheorem{remark}[algorithm]{Remark}

\input{tcilatex.tex}
\input{tcilatex}
\begin{document}
\title{On Jacobi Field Splitting Theorems}
\author{Dennis Gumaer}
\address{Department of Natural Sciences and Mathematics, West Liberty
University}
\email[Dennis Gumaer]{dennis.gumaer@westliberty.edu}
\author{Frederick Wilhelm}
\address{Department of Mathematics, Univ.of Calf., Riverside}
\email[Frederick Wilhelm]{fred@math.ucr.edu }
\urladdr{http://mathdept.ucr.edu/faculty/wilhelm.html}
\date{February 24, 2014}
\subjclass{53C20}
\keywords{Jacobi Fields, Splitting Theorems}

\begin{abstract}
We formulate extensions of Wilking's Jacobi field splitting theorem to
uniformly positive sectional curvature and also to positive and nonnegative
intermediate Ricci curvatures.
\end{abstract}

\maketitle

In \cite{wilking}, Wilking established the following remarkable Jacobi field
splitting theorem.

\begin{theoremlet}
\label{Wilking split}(Wilking) Let $\gamma $ be a unit speed geodesic in a
complete Riemannian $n$--manifold $M$ with nonnegative curvature$.$ Let $%
\Lambda $ be an $(n-1)$-dimensional space of Jacobi fields orthogonal to $%
\gamma $ on which the Riccati operator $S$ is self-adjoint. Then $\Lambda $
splits orthogonally into 
\begin{equation*}
\Lambda =\mathrm{span}\{J\in \Lambda \ |\ J(t)=0\text{ for some }t\}\oplus
\{J\in \Lambda \ |\ J\text{ is parallel}\}.
\end{equation*}
\end{theoremlet}

This result has several impressive applications, so it is natural to ask
about analogs for other curvature conditions. We provide these analogs for
positive sectional curvature and also for nonnegative and positive
intermediate Ricci curvatures. For positive curvature our result is the
following.

\begin{theoremlet}
\label{main} Let $M$ be a complete $n$-dimensional Riemannian manifold with $%
\sec \geq 1$. For $\alpha \in \lbrack 0,\pi ),$ let $\gamma :\left[ \alpha
,\pi \right] \longrightarrow M$ be a unit speed geodesic, and let $\Lambda $
be an $(n-1)$-dimensional family of Jacobi fields orthogonal to $\gamma $ on
which the Riccati operator $S$ is self-adjoint. If 
\begin{equation*}
\max \{\text{eigenvalue }S(\alpha )\}\leq \cot \alpha ,
\end{equation*}%
then $\Lambda $ splits orthogonally into%
\begin{equation}
\mathrm{span}\{J\in \Lambda \ |\ J\left( t\right) =0\text{ for some }t\in
\left( \alpha ,\pi \right) \}\oplus \{J\in \Lambda \ |J=\sin (t)E(t)\text{
with }E\text{ parallel}\}.  \label{splitting displ}
\end{equation}
\end{theoremlet}

Notice that for $\alpha =0,$ the boundary inequality, $\max \{$eigenvalue $%
S(\alpha )\}\leq \cot \alpha =\infty ,$ is always satisfied. So Bonnet's
theorem follows as a corollary.

To understand the initial value hypothesis, $\max \{$eigenvalue $S(\alpha
)\}\leq \cot \alpha ,$ we note that in the model case when $M$ is the unit $%
n $--sphere and $\Lambda =\mathrm{span}\left\{ J|J\left( 0\right) =0\right\}
,$ $S\left( t\right) =\cot \left( t\right) \cdot id.$ So Theorem \ref{main}
says that if $\max \{$eigenvalue $S(\alpha )\}$ is smaller than in the
model, then 
\begin{eqnarray*}
\Lambda &=&\mathrm{span}\{J\in \Lambda \ |\ J\left( t\right) =0\text{ for
some }t\in \left( \alpha ,\pi \right] \} \\
&=&\mathrm{span}\{J\in \Lambda \ |\ J(t)=0\text{ for some }t\in (\alpha ,\pi
)\}\oplus \{J\in \Lambda \ |J=\sin (t)E(t)\text{ with }E\text{ parallel}\}.
\end{eqnarray*}

Like Theorem \ref{Wilking split}, and in contrast to the Rauch Comparison
Theorem, Theorem \ref{main} has no hypothesis about conjugate points.
Geodesics in Complex, Quaternionic, and Octonionic projective spaces with
their canonical metrics provide explicit examples of Theorem \ref{main} in
the presence of conjugate points and also show that both summands in (\ref%
{splitting displ}) can be nontrivial.

Examples \ref{self adj ex} and \ref{boundary inequality is needed} (below)
show that Theorem \ref{main} is optimal in the sense that neither the
boundary inequality nor the hypothesis that $S|_{\Lambda }$ is self-adjoint
can be removed from the statement.

Recall (\cite{Wu}, \cite{Shen}) that a Riemannian manifold $M$ is said to
have $k^{th}$-Ricci curvature $\geq \iota $ provided that for any choice $%
\{v,w_{1},w_{2},\ldots ,w_{k}\}$ of an orthonormal $(k+1)$-frame, the sum of
sectional curvatures $\Sigma _{i=1}^{k}sec(v,w_{i})$ is $\geq \iota $. In
short hand, this is written as $Ric_{k}\ M\geq \iota $. Clearly $Ric_{k}\
M\geq \iota k$ implies $Ric_{k+1}\ M\geq \iota (k+1)$. $Ric_{1}\ M\geq \iota 
$ is the same as $sec\ M\geq \iota $, and $Ric_{n-1}\ M\geq \iota $ is the
same as $Ric\ M\geq \iota $.

When the hypothesis $\sec _{M}\geq 0$ is replaced with $Ric_{k}(M)\geq 0$
and 
\begin{equation*}
dim\{J\in \Lambda \ |\ J(t)=0\text{ for some }t\}\leq n-k-1,
\end{equation*}%
we get the following interpolation between Wilking's Jacobi field splitting
theorem and Theorem 1.7.1 in \cite{gw}.{\Huge \ }

\begin{theoremlet}
\label{Ric_k splitting}Let $\gamma $ be a unit speed geodesic in a complete
Riemannian $n$-manifold $M$ with $Ric_{k}(\dot{\gamma})\geq 0.$ Let $\Lambda 
$ be an $(n-1)$-dimensional space of Jacobi fields orthogonal to $\gamma $
on which the Riccati operator $S$ is self-adjoint. If%
\begin{equation}
dim\left\{ \mathrm{span}\{J\in \Lambda \ |\ J(t)=0\text{ for some }%
t\}\right\} \leq n-k-1,  \label{dim ineq Ineq}
\end{equation}%
then $\Lambda $ splits orthogonally into%
\begin{equation*}
\Lambda =\mathrm{span}\{J\in \Lambda \ |\ J(t)=0\text{ for some }t\}\oplus
\{J\in \Lambda \ |\ J\text{ is parallel}\}.
\end{equation*}%
In particular, if we also have $Ric_{k}(\dot{\gamma})>0,$ then 
\begin{equation*}
dim\left\{ \mathrm{span}\{J\in \Lambda \ |\ J(t)=0\text{ for some }%
t\}\right\} \geq n-k.
\end{equation*}
\end{theoremlet}

For sectional curvature, the Hypothesis (\ref{dim ineq Ineq}) is 
\begin{equation*}
dim\left\{ \mathrm{span}\{J\in \Lambda \ |\ J(t)=0\text{ for some }%
t\}\right\} \leq n-2.
\end{equation*}%
If this is not satisfied, we have $dim\left\{ \mathrm{span}\{J\in \Lambda \
|\ J(t)=0\text{ for some }t\}\right\} =n-1,$ or 
\begin{equation*}
\Lambda =\mathrm{span}\{J\in \Lambda \ |\ J(t)=0\text{ for some }t\}.
\end{equation*}%
Thus Theorem \ref{Ric_k splitting} extends Theorem \ref{Wilking split} to
the $Ric_{k}$ case.

In the Ricci curvature case, Hypothesis (\ref{dim ineq Ineq}) says that all
nonzero Jacobi fields along $\gamma $ are nowhere vanishing, so for Ricci
curvature, Theorem \ref{Ric_k splitting} becomes the following result from 
\cite{gw}.

\begin{theoremlet}
\label{inf splitting thm}(Theorem 1.7.1 in \cite{gw}) Let $\gamma $ be a
unit speed geodesic in a complete Riemannian $n$--manifold $M$ with
nonnegative Ricci curvature$.$ Let $\Lambda $ be an $(n-1)$-dimensional
space of Jacobi fields orthogonal to $\gamma $ on which the Riccati operator 
$S$ is self-adjoint. Let $L\left( t\right) \equiv \{J(t)|J\in \Lambda \}.$
If $L\left( t\right) $ spans $\dot{\gamma}(t)^{\perp }$ for all $t\in 
\mathbb{R}$, then $S\equiv 0$, $\sec (\dot{\gamma},\cdot )\equiv 0$, and $%
\Lambda $ consists of parallel Jacobi fields.
\end{theoremlet}

Just as Wilking applied Theorem \ref{Wilking split} to obtain a structure
result for metric foliations in positive curvature, we apply Theorem \ref%
{Ric_k splitting} to obtain a structure result for metric foliations of
manifolds with positive $Ric_{k}.$ First we recall the definitions of metric
foliations and their dual leaves from \cite{gw} and \cite{wilking}.

\begin{definitionlet}
A metric foliation $\mathcal{F}$ of a Riemannian manifold $M$ is a partition
of $M$ into connected subsets, called leaves, that are locally equidistant
in the following sense. For all $p\in M,$ there are neighborhoods $U\subset
V $ of $p$ so that for any two leaves $L_{1}$ and $L_{2}$ and connected
components $N_{1}$ and $N_{2}$ of $L_{1}\cap V$ and $L_{2}\cap V$,
respectively, the function 
\begin{eqnarray*}
\mathrm{dist}_{N_{1}} &\mathrm{:}&N_{2}\cap U\longrightarrow \mathbb{R}, \\
\mathrm{dist}_{N_{1}}\left( q\right) &\equiv &\mathrm{dist}\left(
q,N_{1}\right)
\end{eqnarray*}%
is constant.
\end{definitionlet}

Examples include the fiber decomposition of a Riemannian submersion and the
orbit decomposition of an isometric group action.

A unit speed geodesic $\gamma :\left( 0,\infty \right) \longrightarrow M$ is
called horizontal for $\mathcal{F}$ if and only if for all $t_{0}>0,$ there
is an $\varepsilon >0$ so that 
\begin{equation*}
\mathrm{dist}\left( L\left( \gamma \left( t_{0}\right) \right) ,\gamma
\left( t\right) \right) =\left\vert t-t_{0}\right\vert
\end{equation*}%
for all $t\in \left( t_{0}-\varepsilon ,t_{0}+\varepsilon \right) .$ Here $%
L\left( \gamma \left( t_{0}\right) \right) $ is the leaf containing $\gamma
\left( t_{0}\right) .$

The \emph{Dual Leaf through }$p$ is defined to be 
\begin{equation*}
\mathfrak{L}^{\#}\left( p\right) \equiv \left\{ q\in M|\text{ there is a
piece-wise smooth horizontal curve from }p\text{ to }q\right\} .
\end{equation*}

\begin{theoremlet}
\label{Ricci_k Submersion}Let $\mathcal{F}$ be a Riemannian foliation of a
complete Riemannian $n$--manifold $E$ with $Ric_{k}(E)>0.$ Then the
dimension of the leaves of $\mathcal{F}^{\#}$ are all $\geq n-k+1.$
\end{theoremlet}

Sectional curvature is the case when $k=1.$ So our conclusion is that the
dimension of the leaves of $\mathcal{F}^{\#}$ are all $\geq n.$ In other
words, there is only one leaf as shown by Wilking in Theorem 1 of \cite%
{wilking}.

For the Ricci curvature case, our conclusion is that the dimension of the
leaves of $\mathcal{F}^{\#}$ are all $\geq 2.$ For an alternative proof of
this fact, combine Theorem \ref{inf splitting thm} with the argument on the
top of page $1305$ of \cite{wilking}.

\begin{remarknonum}
Let $\mathcal{F}$ be a Riemannian foliation of a complete nonnegatively
curved manifold. Wilking showed that the dual foliation of $\mathcal{F}$ is
also Riemannian if its leaves are complete. One might speculate that the
same holds for Riemannian foliations of manifolds with $Ric_{k}\geq 0$ if
the leaves of $\mathcal{F}^{\#}$ are complete and their dimensions are all $%
\leq n-k.$ Wilking's proof uses the Rauch Comparison Theorem in a crucial
way and hence is not directly applicable to the $Ric_{k}$ case.
\end{remarknonum}

For uniformly positive intermediate Ricci curvature we will prove the
following analog of Theorem \ref{main}.

\begin{theoremlet}
\label{Ric_k > k thm} Let $M$ be an $n$-dimensional, complete Riemannian
manifold with $Ric_{k}\geq k$. For $\alpha \in \lbrack 0,\pi ),$ let $\gamma
:\left[ \alpha ,\pi \right] \longrightarrow M$ be a unit speed geodesic. Let 
$\Lambda $ be an $(n-1)$-dimensional family of Jacobi fields on which the
Riccati operator $S$ is self-adjoint. If 
\begin{equation*}
\max \{\text{eigenvalue }S(\alpha )\}\leq \cot \alpha
\end{equation*}%
and if 
\begin{equation*}
dim\{J\in \Lambda \ |\ J(t)=0\text{ for some }t\}\leq n-k-1,
\end{equation*}%
then $\Lambda $ splits orthogonally into 
\begin{equation*}
\mathrm{span}\{J\in \Lambda \ |\ J(t)=0\text{ for some }t\in (\alpha ,\pi
)\}\oplus \{J\in \Lambda \ |\ J=\sin (t)E(t)\text{ with }E\text{ parallel}\}.
\end{equation*}
\end{theoremlet}

In contrast to Ricci curvature, the $Ric_{k}$ condition has the disadvantage
that it is not given by a tensor. On the other hand, constructing examples
with positive $Ric_{k}$ with existing techniques seems much easier than
constructing examples with positive sectional curvature. For example, many
more normal homogeneous spaces have $Ric_{k}>0$ than have positive sectional
curvature.

The proof of Theorem \ref{main} begins with the following extension of \ref%
{inf splitting thm} to positive Ricci curvature.

\begin{theoremlet}
\label{pos}Let $\gamma $ be a unit speed geodesic in a complete Riemannian $%
n $--manifold $M$ with $Ric\geq n-1.$ For $\alpha \in \left[ 0,\pi \right) ,$
let $\Lambda $ be an $(n-1)$-dimensional family of Jacobi fields orthogonal
to a unit speed geodesic $\gamma :\left[ \alpha ,\pi \right] \rightarrow M$
on which the associated Riccati operator $S$ is self-adjoint. Let $L\equiv
\{J(t)|J\in \Lambda \}$ and assume that $L$ spans $\dot{\gamma}(t)^{\perp }$
for $t\in (\alpha ,\pi )$. If 
\begin{equation*}
\max \{\text{eigenvalue }S\left( \alpha \right) \}\leq \cot \alpha ,
\end{equation*}%
then $S\equiv \cot (t)\cdot id$ and consequently $\sec (\dot{\gamma},\cdot
)\equiv 1$.
\end{theoremlet}

Notice that for $\alpha =0,$ the boundary inequality, $\max \{$eigenvalue $%
S(\alpha )\}\leq \cot \alpha =\infty ,$ is always satisfied. So Myer's
Theorem follows as a corollary. Just as we can view Theorem \ref{inf
splitting thm} as an infinitesimal version of the Splitting Theorem \cite%
{CheegGrom}, Theorem \ref{pos} can be viewed as an infinitesimal version of
Cheng's maximal diameter theorem \cite{cheng}.

Section 1 begins with the proofs of Theorems \ref{main} and \ref{pos} and
concludes with examples that show that Theorem \ref{main} is optimal in the
sense that neither the boundary inequality nor the hypothesis that $%
S|_{\Lambda }$ is self-adjoint can be removed from the statement. The proofs
of Theorems \ref{Ric_k splitting}, \ref{Ricci_k Submersion}, and \ref{Ric_k
> k thm} are given in Section 2.

\begin{remarknonum}
See \cite{VerdZil} for other interesting extensions of Wilking's Jacobi
field results.

It seems that Theorem \ref{Ric_k splitting} could be derived from Theorem B
of \cite{VerdZil}, but since our proof of Theorem \ref{Ric_k splitting} is
so simple, we have not studied its relationship to Theorem B of \cite%
{VerdZil} in detail. It might also be possible to prove our Theorem \ref%
{main} using the techniques of \cite{VerdZil}.
\end{remarknonum}

\begin{acknowledgment}
We are grateful to the referee for a thoughtful critique of the manuscript.  
\end{acknowledgment}

\section{Uniformly Positive Curvature}

\numberwithin{equation}{algorithm}

First we work on the proof of Theorem \ref{pos}. Set 
\begin{eqnarray*}
s &\equiv &\frac{1}{n-1}trace\left( S\right) , \\
S_{0} &\equiv &S-\frac{trace\left( S\right) }{n-1}\cdot id,\text{ } \\
r &\equiv &\frac{1}{n-1}\left( Ric\left( \dot{\gamma},\dot{\gamma}\right)
+\left\vert S_{0}\right\vert ^{2}\right) ,\text{ and}
\end{eqnarray*}%
abusing notation, let%
\begin{equation*}
R\left( \cdot \right) =R\left( \cdot ,\dot{\gamma}\right) \dot{\gamma}.
\end{equation*}%
Recall (\cite{gw}, page 36) that the Jacobi equation is equivalent to the
two first order equations 
\begin{eqnarray}
S\left( J\right) &=&J^{\prime }\text{ and}  \notag \\
S^{2}+S^{\prime }+R &=&0,  \label{Ricatti}
\end{eqnarray}%
and that Equation \ref{Ricatti} implies%
\begin{equation*}
s^{2}+s^{\prime }+r=0.
\end{equation*}

\begin{proposition}
\label{proppos}Assume the hypotheses of Theorem \ref{pos}. If $s(t)\equiv
\cot (t),$ then $S(t)\equiv \cot (t)\cdot id$ and therefore $\sec (\gamma
^{\prime }(t),\cdot )\equiv 1$.
\end{proposition}

\begin{proof}
Substituting $s(t)\equiv \cot (t)$ into $s^{2}+s^{\prime }+r=0$ gives 
\begin{align*}
0& =\cot ^{2}(t)-\csc ^{2}(t)+r \\
& =-1+r.
\end{align*}%
So $r\equiv 1$. Consequently 
\begin{align*}
1& =r \\
& =\frac{Ric\left( \dot{\gamma},\dot{\gamma}\right) +|S_{0}|^{2}}{n-1} \\
& \geq \frac{(n-1)+|S_{0}|^{2}}{n-1} \\
& =1+\frac{|S_{0}|^{2}}{n-1}.
\end{align*}

Thus $|S_{0}|\equiv 0$, and 
\begin{eqnarray*}
S &=&\frac{trace\left( S\right) }{n-1}\cdot id \\
&=&s\cdot id \\
&=&\cot (t)\cdot id.
\end{eqnarray*}%
Substituting $S=\cot (t)\cdot id$ into the Riccati equation, $%
S^{2}+S^{\prime }+R=0,$ gives%
\begin{eqnarray*}
\left( \cot ^{2}(t)-\csc ^{2}(t)\right) \cdot id+R &=&0, \\
-id+R &=&0,
\end{eqnarray*}%
and therefore $\sec (\gamma ^{\prime },\cdot )\equiv 1$.
\end{proof}

\begin{remark}
\label{modelpos}The solution to the initial value problem 
\begin{equation*}
f^{2}+f^{\prime }+1=0;\qquad f(t_{0})=s(t_{0}),\text{ }t_{0}\in \left[ 0,\pi %
\right]
\end{equation*}%
is 
\begin{equation*}
f(t)=\cot \left( t-\left( t_{0}-\cot ^{-1}s\left( t_{0}\right) \right)
\right) .
\end{equation*}%
If $s(t_{0})<\cot (t_{0}),$ it follows that $f$ has an asymptote at a time $%
t=d\in (t_{0},\pi ),$ and $\lim_{t\rightarrow d^{-}}f\left( t\right)
=-\infty .$ If $s(t_{0})>\cot (t_{0})$, it follows that $f$ has an asymptote
at a time $t=d\in (0,t_{0}),$ and $\lim_{t\rightarrow d+}f\left( t\right)
=\infty .$
\end{remark}

\begin{proposition}
\label{xandypos} Assume the hypotheses of Theorem \ref{pos}. Let $f$ be as
in Remark \ref{modelpos} for some $t_{0}\in \left[ \alpha ,\pi \right] .$
Then

\begin{itemize}
\item[(a)] $s|_{[\alpha ,t_{0})}\geq f|_{[\alpha ,t_{0})}$ and

\item[(b)] $s|_{(t_{0},\pi )}\leq f|_{(t_{0},\pi )}$.
\end{itemize}
\end{proposition}

\begin{proof}
Set $y=f-s$ on $(\alpha ,t_{0})$. Using $f^{2}+f^{\prime }+1=0$ and $%
s^{2}+s^{\prime }+r=0$ we find 
\begin{align*}
y^{\prime }& =f^{\prime }-s^{\prime } \\
& =-f^{2}+s^{2}+r-1 \\
& =-(f-s)(f+s)+r-1.
\end{align*}%
So $y$ solves the initial value problem 
\begin{equation}
y^{\prime }=-(f+s)y+r-1,\qquad y(t_{0})=0.  \label{y IV}
\end{equation}

Let $x$ be a nontrivial solution to $x^{\prime }=-\frac{1}{2}(f+s)x$ and $u$
a function that satisfies $u^{\prime }=\frac{2\left( r-1\right) }{x^{2}}$
and $u(t_{0})=0$. We claim that this implies $y=\frac{1}{2}ux^{2}.$ To
justify this claim, we'll show $\frac{1}{2}ux^{2}$ solves the initial value
problem \ref{y IV}. Indeed, 
\begin{eqnarray}
\left( \frac{1}{2}ux^{2}\right) ^{\prime } &=&\frac{1}{2}\left( u^{\prime
}x^{2}+2uxx^{\prime }\right)  \notag \\
&=&\frac{1}{2}\left( \frac{2\left( r-1\right) }{x^{2}}x^{2}-ux(f+s)x\right) 
\notag \\
&=&\left( r-1\right) -\frac{1}{2}ux^{2}(f+s).  \notag
\end{eqnarray}%
Since $u\left( t_{0}\right) =0,$ our claim holds and $y=\frac{1}{2}ux^{2}.$

Since $r\geq 1,$ $u^{\prime }=\frac{2\left( r-1\right) }{x^{2}}\geq 0.$
Combined with $u(t_{0})=0,$ we see that $u\leq 0$ on $\left[ \alpha
,t_{0}\right) .$ Hence $y|_{\left[ \alpha ,t_{0}\right) }=\frac{1}{2}%
ux^{2}|_{\left[ \alpha ,t_{0}\right) }\leq 0$, and $s|_{\left[ \alpha
,t_{0}\right) }\geq f|_{\left[ \alpha ,t_{0}\right) },$ proving Part $a.$

Similarly, $u^{\prime }\geq 0$ and $u(t_{0})=0$ yield $u|_{(t_{0},\pi )}\geq
0.$ Hence $s|_{(t_{0},\pi )}\leq f|_{(t_{0},\pi )}$ as claimed.
\end{proof}

We are now ready to prove Theorem \ref{pos}.

\begin{proof}[Proof of Theorem \protect\ref{pos}]
Assume, for contradiction, that $S\not\equiv \cot (t)\cdot id$. It follows
from Proposition \ref{proppos} that $s(t)\not\equiv \cot (t)$.

Assume for the moment that $s(t_{0})<\cot (t_{0})$ for some $t_{0}\in \left[
\alpha ,\pi \right] .$ Then by Remark \ref{modelpos}, $f$ has an asymptote
at a time $d\in (t_{0},\pi ),$ and $\lim_{t\rightarrow d^{-}}f\left(
t\right) =-\infty .$ On the other hand, by Proposition \ref{xandypos}(b), $%
s|_{(t_{0},\pi )}\leq f|_{(t_{0},\pi )}$. So $s$ is not defined for all $%
t\in (\alpha ,\pi )$. This contradicts our hypothesis that $L\equiv \left\{
J(t)|J\in \Lambda \right\} $ spans $\dot{\gamma}(t)^{\perp }$ for $t\in
(\alpha ,\pi ).$

We may assume therefore $s(t)\geq \cot (t)$ on $[\alpha ,\pi )$. Since we
assumed 
\begin{equation*}
\max \{\text{eigenvalue }S\left( \alpha \right) \}\leq \cot \alpha ,
\end{equation*}%
it follows that $s(\alpha )=\cot (\alpha ).$ So either, $s\equiv \cot $ or
for some $t_{0}\in \left( \alpha ,\pi \right] $, $s(t_{0})>\cot (t_{0})$.
Assume the latter. Then by Remark \ref{modelpos}, the graph of $f$ has an
asymptote on $(0,t_{0})$. If the asymptote for $f$ lies in $(0,\alpha ),$
then the graph of $f$ is a shift to the right of the graph of $\cot $ by an
amount in $\left( 0,\alpha \right) .$ Combining this with Proposition \ref%
{xandypos}(a) gives 
\begin{equation*}
s(\alpha )\geq f(\alpha )>\cot (\alpha )=s(\alpha ),
\end{equation*}%
a contradiction.

The remaining case is when $s(t_{0})>\cot (t_{0})$ and the asymptote for $f$
is in $(\alpha ,t_{0})$. By Proposition \ref{xandypos}(a), 
\begin{equation*}
s|_{[\alpha ,t_{0})}\geq f|_{[\alpha ,t_{0})},
\end{equation*}%
and by Remark \ref{modelpos}, 
\begin{equation*}
\lim_{t\rightarrow d+}f\left( t\right) =\infty .
\end{equation*}%
Together the previous two displays imply that $s$ is not defined at $d,$
contrary to our hypothesis that $L\equiv \left\{ J(t)|J\in \Lambda \right\} $
spans $\dot{\gamma}(t)^{\perp }$ for $t\in (\alpha ,\pi ).$ Hence $s\equiv
\cot $, and by Proposition \ref{proppos}, $S\equiv \cot (t)\cdot id$ and $%
\sec (\gamma ^{\prime }(t),\cdot )\equiv 1$ as desired.
\end{proof}

We complete the proof of Theorem \ref{main} by mostly following the lines of
Wilking's proof of Theorem \ref{Wilking split}. In particular, we use
Wilking's generalization of the Horizontal Curvature Equation from \cite%
{wilking}, which we review in outline below. For details see \cite{gw}.

Let $\gamma $ be a unit speed geodesic in a complete Riemannian $n$%
--manifold $M.$ Let $\Lambda $ be an $(n-1)$-dimensional space of Jacobi
fields orthogonal to $\gamma $ on which the Riccati operator $S$ is
self-adjoint. Let $\Psi $ be any vector subspace of $\Lambda .$ Define 
\begin{equation*}
V(t)\equiv \{J(t)\ |\ J\in \Psi \}\oplus \{J^{\prime }(t)\ |\ J\in \Psi
,J(t)=0\}.
\end{equation*}%
Note that the second summand vanishes for almost every $t,$ and $V(t)$
defines a smooth distribution along $\gamma $ (\cite{gw}, Lemma 1.7.1, \cite%
{wilking}, page 1300). Set 
\begin{equation*}
H(t)\equiv V(t)^{\perp }\cap \dot{\gamma}(t)^{\perp }.
\end{equation*}%
At a time $t_{0}$ that satisfies $V(t_{0})=\{J(t_{0})\ |\ J\in \Psi \},$ we
define a Riccati operator 
\begin{eqnarray*}
\hat{S} &:&H\left( t_{0}\right) \longrightarrow H\left( t_{0}\right) \text{
by} \\
\hat{S}\left( y\right) &=&\left( \left( J^{h}\right) ^{\prime }\right)
^{h}|_{t_{0}},\text{ where }J\in \Lambda \text{ satisfies }J\left(
t_{0}\right) =y,
\end{eqnarray*}%
and the superscript $^{h}$ denotes the component in $H.$ We also define a
map 
\begin{eqnarray*}
A &:&V\left( t_{0}\right) \longrightarrow H\left( t_{0}\right) \text{ by} \\
A\left( u\right) &=&\left( J^{\prime }\right) ^{h}\left( t_{0}\right) ,\text{
where }J\in \Psi ,\text{ }J\left( t_{0}\right) =u.
\end{eqnarray*}

\begin{theorem}
(Wilking) \label{HCE thm}$\hat{S}$ and $A$ are well-defined maps for all $%
t\in \mathbb{R},$ $\hat{S}$ is self-adjoint, and $\hat{S}$ and $A$ satisfy
the Riccati type equation 
\begin{equation}
\hat{S}^{\prime }+\hat{S}^{2}+\left\{ R\left( \cdot ,\dot{\gamma}(t)\right) 
\dot{\gamma}(t)\right\} ^{h}+3AA^{\ast }=0.  \label{Horiz Wilki}
\end{equation}
\end{theorem}

\begin{remark}
To see how this generalizes the Horizontal Curvature Equation of \cite{Gray}
and \cite{O'Neill}, suppose $\gamma $ is a horizontal geodesic for a
Riemannian submersion $\pi :M\longrightarrow B.$ For $\Lambda $ take the
Jacobi fields that correspond to variations of horizontal geodesics that
leave $\pi ^{-1}\left( \pi \left( \gamma \left( 0\right) \right) \right) $
orthogonally. For $\Psi $ take the Holonomy fields, that is those that come
from the lifts of $\gamma .$ Then via the Horizontal Curvature Equation,
Equation \ref{Horiz Wilki} becomes the usual Riccati Equation along $\pi
\left( \gamma \right) .$
\end{remark}

\begin{proof}[Proof of Theorem \protect\ref{main}]
Let $M$ be an $n$-dimensional Riemannian manifold with $\sec \geq 1$. For $%
\alpha \in \lbrack 0,\pi ),$ let $\gamma :\left[ \alpha ,\pi \right]
\longrightarrow M$ be a unit speed geodesic. Let $\Lambda $ be an $(n-1)$%
-dimensional family of Jacobi fields on which the Riccati operator $S$ is
self-adjoint and satisfies 
\begin{equation*}
\max \{\text{eigenvalue }S(\alpha )\}\leq \cot \alpha .
\end{equation*}

Set 
\begin{equation*}
\Psi \equiv \{J\in \Lambda \ |\ J(t)=0\text{ for some }t\},
\end{equation*}%
and define $V\left( t\right) ,$ $H\left( t\right) ,$ the Riccati operator $%
\hat{S}:H\left( t\right) \longrightarrow H\left( t\right) $ and the map 
\newline
$A:V\left( t\right) \longrightarrow H\left( t\right) $ as above.

By Theorem \ref{Horiz Wilki}, $\hat{S}$ and $A$ satisfy the Riccati-type
equation 
\begin{equation}
\hat{S}^{\prime }+\hat{S}^{2}+\left\{ R\left( \cdot ,\dot{\gamma}(t)\right) 
\dot{\gamma}(t)\right\} ^{h}+3AA^{\ast }=0.  \label{hoirz wilking}
\end{equation}

Next we show that 
\begin{equation*}
\max \{\text{eigenvalue }\hat{S}(\alpha )\}\leq \cot \alpha .
\end{equation*}

For $z\in H(\alpha ),$ let $J\in \Lambda $ satisfy $J\left( \alpha \right)
=z.$ Then%
\begin{eqnarray}
g\left( \hat{S}(\alpha )z,z\right) &=&g\left( \left( \left( J^{h}\right)
^{\prime }\right) ^{h},z\right)  \notag \\
&=&g\left( \left( J^{h}\right) ^{\prime },z\right) .  \label{hat(S)}
\end{eqnarray}%
If $J^{v}\left( t\right) $ denotes the component of $J\left( t\right) $ in $%
V\left( t\right) ,$ then for $X$ tangent to $H,$ 
\begin{equation*}
g\left( \left( J^{v}\right) ^{\prime },X\right) |_{\alpha }=-g\left(
J^{v},X^{\prime }\right) |_{\alpha }=0,
\end{equation*}%
since $J^{v}\left( \alpha \right) =0.$ In particular, $\left( \left(
J^{v}\right) ^{\prime }\right) ^{h}|_{\alpha }=0$. Combined with Equation %
\ref{hat(S)} this gives%
\begin{eqnarray*}
g\left( \hat{S}(\alpha )z,z\right) &=&g\left( J^{\prime },z\right) \\
&=&g\left( S(\alpha )z,z\right) \\
&\leq &g\left( \cot (\alpha )z,z\right) .
\end{eqnarray*}%
Hence $\max \{\text{eigenvalue }\hat{S}(\alpha )\}\leq \cot \alpha ,$ as
claimed.

Now set 
\begin{equation*}
\hat{R}\equiv \left\{ R\left( \cdot ,\dot{\gamma}(t)\right) \dot{\gamma}%
(t)\right\} ^{h}+3AA^{\ast },
\end{equation*}%
and for $x\in H\left( t\right) ,$ 
\begin{equation*}
\widehat{\text{\textrm{sec}}}\left( x\right) \equiv g\left( \hat{R}\left( 
\frac{x}{\left\vert x\right\vert }\right) ,\frac{x}{\left\vert x\right\vert }%
\right) .
\end{equation*}%
Since $AA^{\ast }$ is a nonnegative operator and $\sec \left( M\right) \geq
1,$ $\widehat{\text{\textrm{sec}}}\geq 1.$ Substituting into Equation \ref%
{hoirz wilking} we have 
\begin{equation*}
\hat{S}^{\prime }+\hat{S}^{2}+\hat{R}=0.
\end{equation*}

Since $\widehat{\text{\textrm{sec}}}\geq 1$ and $\max \{$eigenvalue $\hat{S}%
(\alpha )\}\leq \cot \alpha ,$ we apply the proof of Theorem \ref{pos} and
conclude that $\hat{S}\equiv \cot (t)\cdot id$ and consequently $\widehat{%
\text{\textrm{sec}}}\equiv 1.$ Substituting back into $\hat{S}^{\prime }+%
\hat{S}^{2}+\hat{R}=0,$ we find

\begin{equation*}
-\csc ^{2}(t)\cdot id+\cot ^{2}\cdot id+\left\{ R\left( \cdot ,\dot{\gamma}%
(t)\right) \dot{\gamma}(t)\right\} ^{h}+3AA^{\ast }=0
\end{equation*}%
or%
\begin{equation*}
\left\{ R\left( \cdot ,\dot{\gamma}(t)\right) \dot{\gamma}(t)\right\}
^{h}+3AA^{\ast }=id.
\end{equation*}%
Applying both sides to an $x\in H\left( t\right) $ and then taking the inner
product with $x$ gives 
\begin{eqnarray*}
g\left( x,x\right) &=&g\left( R\left( x,\dot{\gamma}(t)\right) \dot{\gamma}%
(t),x\right) +g\left( 3AA^{\ast }\left( x\right) ,x\right) \\
&\geq &g\left( x,x\right) +g\left( 3AA^{\ast }\left( x\right) ,x\right) .
\end{eqnarray*}%
Thus%
\begin{equation*}
0\geq g\left( 3AA^{\ast }\left( x\right) ,x\right) .
\end{equation*}%
On the other hand, $g\left( AA^{\ast }\left( x\right) ,x\right) =g\left(
A^{\ast }\left( x\right) ,A^{\ast }\left( x\right) \right) \geq 0;$ so $%
AA^{\ast }x\equiv 0,$ and $A\equiv 0.$ It follows that $H\left( t\right) $
is parallel, and $\mathrm{sec}\left( x,\dot{\gamma}\right) \equiv 1$ for all 
$x\in H\left( t\right) .$ So $H\left( t\right) $ is spanned by Jacobi fields
of the form $\sin (t)E(t)$ where $E$ is a parallel field tangent to $H.$
\end{proof}

\subsection*{Jacobi Fields That Violate Our Hypotheses}

The following example shows that the hypothesis in Theorem \ref{main} about
the Riccati operator being self-adjoint is necessary.

\begin{example}
\label{self adj ex}Consider $S^{3}$ with the round metric. Define 
\begin{equation*}
J_{1}(t)\equiv \sin (t)\cdot E_{1}(t)+\cos (t)\cdot E_{2}(t)
\end{equation*}%
and 
\begin{equation*}
J_{2}(t)\equiv \cos (t)\cdot E_{1}(t)-\sin (t)\cdot E_{2}(t)
\end{equation*}%
where $E_{1}(t)$ and $E_{2}(t)$ are orthonormal parallel fields. Set $%
\Lambda \equiv $\textrm{$span$}$\left\{ J_{1},J_{2}\right\} .$ Since $%
\left\{ J_{1}\left( t\right) ,J_{2}\left( t\right) \right\} $ is orthonormal
for all $t,$ no Jacobi field $J\in \Lambda $ has a zero, so the conclusion
of Theorem \ref{main} does not hold for this family. On the other hand,
neither do the hypotheses, since $S|_{\Lambda }$ is not self-adjoint.
Indeed, 
\begin{align*}
g\left( SJ_{1},J_{2}\right) & =g\left( \cos (t)\cdot E_{1}(t)-\sin (t)\cdot
E_{2}(t),\cos (t)\cdot E_{1}(t)-\sin (t)\cdot E_{2}(t)\right) \\
& =\cos ^{2}(t)+\sin ^{2}(t) \\
& =1
\end{align*}%
and 
\begin{align*}
g\left( J_{1},SJ_{2}\right) & =g\left( \sin (t)\cdot E_{1}(t)+\cos (t)\cdot
E_{2}(t),-\sin (t)\cdot E_{1}(t)-\cos (t)\cdot E_{2}(t)\right) \\
& =-\sin ^{2}(t)-\cos ^{2}(t) \\
& =-1.
\end{align*}
\end{example}

The requirement that the maximum of the eigenvalues be bounded above is also
necessary, and our hypothesis is the optimal one. This is demonstrated by
the following example.

\begin{example}
\label{boundary inequality is needed}Let $\gamma :\left( -\infty ,\infty
\right) \longrightarrow S^{n}\left( 1\right) $ be a geodesic in the unit
sphere. Let $\left\{ E_{i}\right\} _{i=1}^{n-1}$ be a basis of parallel
fields along $\gamma $ that are all normal to $\gamma .$ For $\varepsilon
\in \left( 0,\frac{\pi }{6}\right) ,$ set 
\begin{equation*}
J_{i}=\sin (t-\varepsilon )\cdot E_{i}(t),
\end{equation*}%
and define $\Lambda \equiv $\textrm{$span$}$\left\{ J_{i}\right\} .$ Since $%
J_{i}\left( \varepsilon \right) =0$ for all $i,$ $S|_{\Lambda }$ is
self-adjoint. Moreover, $S\left( t\right) =\cot \left( t-\varepsilon \right)
id.$ So for $\alpha =\frac{\pi }{2}$ we have%
\begin{equation*}
\max \{\text{eigenvalue }S\left( \frac{\pi }{2}\right) \}=\cot \left( \frac{%
\pi }{2}-\varepsilon \right) >0=\cot \left( \frac{\pi }{2}\right) .
\end{equation*}%
That is our family, $\Lambda $, does not satisfy our boundary hypothesis, $%
\max \{$eigenvalue $S(\alpha )\}\leq \cot \alpha .$ On the other hand, each $%
J\in \Lambda \setminus \left\{ 0\right\} $ is nonvanishing on $\left[ \alpha
,\pi \right] ,$ so the conclusion does not hold either.
\end{example}

\section{Positive Ricci$_{k}$}

\begin{proof}[Proof of Theorem \protect\ref{Ric_k splitting}]
Let 
\begin{equation*}
\mathcal{Z}\equiv \{J\in \Lambda \ |\ J(t)=0\text{ for some }t\}.
\end{equation*}

Suppose 
\begin{equation*}
dim\mathcal{Z}\leq n-1-k.
\end{equation*}%
Then%
\begin{equation*}
dim\mathcal{Z}^{\perp }\geq \left( n-1\right) -\left( n-1-k\right) =k
\end{equation*}%
Apply Wilking's generalization of the horizontal curvature equation, Theorem %
\ref{HCE thm}, and the proof of Theorem \ref{inf splitting thm} to conclude
that $\Lambda $ splits orthogonally into 
\begin{equation*}
\Lambda =\mathcal{Z}\oplus \mathcal{P}
\end{equation*}%
where $\mathcal{P}$ is a family of parallel Jacobi fields and has dimension $%
\geq k.$

In particular $Ric_{k}\left( \dot{\gamma}\right) $ is not positive, so if $%
Ric_{k}\left( \dot{\gamma}\right) $ is known to be positive, then it must be
that 
\begin{equation*}
\mathrm{dim}\mathcal{Z}\geq n-k.
\end{equation*}
\end{proof}

A nearly identical argument gives the proof of Theorem \ref{Ric_k > k thm},
except the role of the proof of Theorem \ref{inf splitting thm} is played by
the proof of Theorem \ref{pos}. To apply Theorem \ref{pos} we need to
establish the hypothesis on the upper bound on the initial values of the
Riccati operator. This is achieved in precisely the same manner as in the
proof of Theorem $\ref{main}.$

\begin{proof}[Proof of Theorem \protect\ref{Ricci_k Submersion}]
Let $\gamma :\left( -\infty ,\infty \right) \longrightarrow E$ be a
horizontal geodesic. Let $\Lambda $ be the space of Jacobi fields along $%
\gamma $ that arise from variations that leave $\pi ^{-1}\left( \pi \left(
\gamma \left( 0\right) \right) \right) $ orthogonally. Let 
\begin{equation*}
\mathcal{Z}\equiv \{J\in \Lambda \ |\ J(t)=0\text{ for some }t\}.
\end{equation*}%
By Theorem \ref{Ric_k splitting}, $dim\mathcal{Z}\geq n-k.$

On the other hand, according to Wilking (top of page $1305$ of \cite{wilking}%
), $\mathcal{Z}$ is everywhere tangent to the dual leaf through $\gamma
\left( 0\right) .$ Since $\mathcal{Z}$ is everywhere orthogonal to $\gamma $
and $\gamma $ is everywhere tangent to the dual leaf, the dimension of the
dual leaf is $\geq n-k+1$.
\end{proof}

\end{document}